\documentclass[a4paper]{article}
\usepackage{fullpage}
\usepackage{tikz}

\usepackage[pdftex,breaklinks,colorlinks,
linkcolor=blue,
citecolor=blue,
urlcolor=blue]{hyperref}
\newcommand{\pr}[1]{{\color{red}#1}}

\usepackage{enumitem}
\usepackage{amsmath,amssymb,amsfonts,amsthm}

\newtheorem{theorem}{Theorem}
\newtheorem{conjecture}[theorem]{Conjecture}
\newtheorem{lemma}[theorem]{Lemma}
\newtheorem{proposition}[theorem]{Proposition}

\theoremstyle{definition}
\newtheorem{definition}[theorem]{Definition}
\newtheorem{remark}[theorem]{Remark}

\author{Pablo Romero\footnote{Facultad de Ingenier\'ia, Universidad de la Rep\'ublica, Montevideo, Uruguay. E-mail address: \texttt{promero@fing.edu.uy}}}

\date{}

\makeatletter%
\begin{document}

\title{An algebraic characterization of strong graphs}

\maketitle

\begin{abstract}\let\thefootnote\relax
Let $G$ be a connected simple graph on $n$ vertices and $m$ edges. 
Denote $N_{i}^{(j)}(G)$ the number of spanning subgraphs of $G$ having precisely $i$ edges and not more than $j$ connected components. 
The graph $G$ is \emph{strong} if $N_{i}^{j}(G)\geq N_{i}^{j}(H)$ for each pair of integers $i\in \{0,1,\ldots,m\}$ and $j\in \{1,2,\ldots,n\}$ and each connected simple graph $H$ on $n$ vertices and $m$ edges. 
The graph $G$ is \emph{Whitney-maximum} if 
for each connected simple graph $H$ on $n$ vertices and $m$ edges there exists a polynomial $P_H(x,y)$ with nonnegative coefficients such that  $W_{G}(x,y)-W_H(x,y)=(1-xy)P_H(x,y)$, where $W_G$ and $W_H$ stand for the Whitney polynomial of $G$ and $H$. 
In this work it is proved that a graph is strong if and only if it is Whitney-maximum. 
Consequently, the $0$-element conjecture proposed by Boesch [J.\ Graph Theory 10 (1986), 339--352] is true when restricted to graph classes in which Whitney-maximum graphs exist. 
\end{abstract}

\renewcommand{\labelitemi}{--}

\section{Historic motivation}\label{section:intro}
During the half of the $20$th century there was a growing interest in two problems that are different in appearance, namely, the $4$-coloring problem and the design of networks with maximum reliability. Some key references in both problems are given in the following paragraphs and the interplay between them is explained,  which is the main source of inspiration for this work.

On the one hand Birkhoff~\cite{Birkhoff-1912}, interested in the $4$-coloring problem,  defined a univariate polynomial known as the chromatic polynomial for planar graphs. Later works by Whitney~\cite{Whitney-1932} and Tutte~\cite{Tutte-1954} considered generalizations of the chromatic polynomial using bivariate polynomials. The work of Tutte~\cite{Tutte-1954} linked the theory of spanning trees, electrical networks, and graph coloring. In his Ph.D. thesis, he proved the universality property of his eponymous polynomial. Essentially, each graph invariant that obeys the deletion-contraction formula can be obtained by a special evaluation of the Tutte polynomial. The interested reader can consult the recent book~\cite{Book-Tutte-2022} for a rich historic account of the Tutte polynomial and an updated reference in this topic.

On the other hand Moore and Shannon~\cite{1956-MooreShannon} wanted to design highly reliable computers using imperfect electronic components such as relays. The work of Moore and Shannon is considered a point of departure in the study of network reliability. The concept of uniformly most reliable graphs was later introduced by Boesch~\cite{1986-Boesch}. 
The existence and construction of uniformly most reliable graphs is a current research topic. Boesch et al.~\cite{1991-Boesch} found all uniformly most reliable graphs of corank up to $3$. 
Curiously enough, the first classes of uniformly most reliable graphs were determined by Kelmans~\cite{1981-Kelmans} some years before the definition was formally established in print. In his work he showed that if we remove an arbitrary matching to the complete graph then a uniformly most reliable graph is obtained. The reader can consult the recent survey~\cite{2022-Romero} for further details.

Even though the classical reliability problem aims to find the probability that a graph is connected subject to edge failures, Kelmans~\cite{1981-Kelmans} studied a much more general problem in which we are given a positive integer $k$ and the aim is to find the probability that a graph has at most $k$ connected components after each of its edges is independently deleted with given probabilities. This generalization due to Kelmans has not been further explored. The counting problems for the study of graphs having multiple connected components are even harder than the ones that appear in the classical reliability problem faced by Boesch, Moore and Shannon.

Joint works of Kahl~\cite{Kahl-2022}, and Kahl and Luttrell~\cite{Kahl-2023}, established links between uniformly most reliable graphs and the Tutte polynomial by means of a novel concept of a \emph{Tutte-maximum graph}. Using the universality property of the Tutte polynomial it follows that Tutte-maximum graphs are uniformly most reliable graphs. Kahl and Luttrell~\cite{Kahl-2023} then proceeded to recover the results  obtained by Boesch et al. on the existence of uniformly most reliable graphs of corank up to $3$ but using algebraic methods. 
Furthermore, Kahl~\cite{Kahl-2022} showed that each Tutte-maximum graph  simultaneously attains the extremal value (i.e., minimum or maximum) of several real-valued graph invariants among all graphs with a prescribed number of vertices and edges.

The goal of this work is to further explore the interplay between counting problems coming from network reliability analysis and algebraic methods that are available using the Tutte polynomial or Whitney polynomial. This article is organized as follows. A background including some concepts and preliminary results is given in Section~\ref{section:background}. The main result of this work is given in Section~\ref{section:main}, where it is proved that a graph is strong if and only if it is Whitney-maximum. In Section~\ref{section:consequences} it is proved that 
the $0$-element conjecture proposed by Boesch~\cite{1986-Boesch} is true when restricted to graph classes in which a Whitney-maximum graph exists. 

\section{Background}\label{section:background}
In this section we include the concept of Tutte-maximum graph~\cite{Kahl-2023} and some concepts on network reliability. We will also include the 0-element conjecture proposed by Boesch~\cite{1986-Boesch} as well as preliminary results which will be used in the sequel.

Let $G$ be a graph on $n$ vertices, $m$ edges and $\kappa(G)$ connected components. 
The \emph{corank of $G$}, denoted $c(G)$, equals $m-n+\kappa(G)$. The \emph{rank of $G$}, denoted $r(G)$, equals $n-\kappa(G)$. Let $\mathcal{S}(G)$ be the set of all spanning subgraphs of $G$. The \emph{Tutte polynomial of $G$} is denoted $T_G(x,y)$ and is defined as follows,
\begin{equation}\label{eq:tutte}
T_G(x,y) = \sum_{H \in \mathcal{S}(G)}(x-1)^{r(G)-r(H)}(y-1)^{c(H)}.   
\end{equation}

The \emph{Whitney polynomial of $G$} is defined as $W_G(x,y)=T_G(x+1,y+1)$, i.e., 
\begin{equation}\label{eq:Whitney}
W_G(x,y) = \sum_{H \in \mathcal{S}(G)}x^{r(G)-r(H)}y^{c(H)}.   
\end{equation}
A bivariate polynomial $P(x,y)$ is \emph{nonnegative} if there exists nonnegative integers $p$ and $q$ and nonnegative real numbers $a_{ij}$ such that $P(x,y)=\sum_{i=0}^p\sum_{j=0}^{q}a_{ij}x^iy^j$. 

Let $\mathcal{C}_{n,m}$ be the set of all connected simple graphs on $n$ vertices and $m$ edges. The concept of Tutte-maximum graphs 
introduced by Kahl and Luttrell~\cite{Kahl-2023} is a point of departure in the construction of algebraic methods to find uniformly most reliable graphs.

\begin{definition}
For each pair of graphs $G$ and $H$ in $\mathcal{C}_{n,m}$ we write $H\preceq G$ 
when $T_G(x,y)-T_{H}(x,y)=(x+y-xy)P_H(x,y)$ for some nonnegative polynomial $P_H(x,y)$. The graph $G$ is \emph{Tutte-maximum} if $H \preceq G$ for each $H$ in $\mathcal{C}_{n,m}$.    
\end{definition}
 
Now, let us present key concepts from network reliability. 
Let $G$ be any graph in $\mathcal{C}_{n,m}$. For each $k\in \{1,2,\ldots,n\}$ and each $p\in [0,1]$, the \emph{$k$-reliability of $G$ at $p$}, denoted $R_{G}^{(k)}(p)$, is the probability that $G$ has at most $k$ connected components after each of its edges is independently retained with probability $p$. 
For each $i \in \{0,1,\ldots,m\}$, let $N_i^{(k)}(G)$ be the number of spanning subgraphs $H$ in $G$ composed by $i$ edges having at most $k$ connected components. Clearly, 
\begin{equation}\label{eq:rk}
R_{G}^{(k)}(p) = \sum_{i=0}^{m}N_i^{(k)}(G)p^i(1-p)^{m-i}.     
\end{equation}

If $k \in \{1,2,\ldots,n\}$ then $G$ is a \emph{$k$-uniformly most reliable graph} ($k$-UMRG) if $R_G^{(k)}(p)\geq R_{H}^{(k)}(p)$ for each $H$ in $\mathcal{C}_{n,m}$ and $p$ in $[0,1]$. The graph $G$ is a \emph{uniformly most reliable graph}~\cite{1986-Boesch} if it is $1$-UMRG. Theorem~\ref{theorem:tutte-reliability} is a consequence of the universality property of the Tutte-polynomial~\cite{Book-Tutte-2022}. 
\begin{theorem}\label{theorem:tutte-reliability}
If $G$ is any graph in $\mathcal{C}_{n,m}$ and $p\in (0,1)$ then 
\begin{equation*}
R_{G}^{(1)}(p) = p^{n-1}(1-p)^{m-n+1}T_G\left(1,\frac{1}{1-p}\right)    
\end{equation*}
\end{theorem}
As corollary, each Tutte-maximum graph is uniformly most reliable~\cite{Kahl-2023}. 

A graph $G$ is a \emph{0-element} in $\mathcal{C}_{n,m}$ if $N_i^{(1)}(G)\geq N_i^{(1)}(H)$ for each $i$ in $\{0,1,\ldots,m\}$ and each $H$ in $\mathcal{C}_{n,m}$. From equation~\eqref{eq:rk} it is clear that each $0$-element is uniformly most reliable. The converse is an unresolved conjecture proposed in 1986 by Boesch~\cite{1986-Boesch}.
\begin{conjecture}[Boesch~\cite{1986-Boesch}]\label{conjecture1}
Each uniformly most reliable graph in $\mathcal{C}_{n,m}$ is a $0$-element in $\mathcal{C}_{n,m}$.
\end{conjecture}
The concept of a $0$-element in $\mathcal{C}_{n,m}$ was introduced by Boesch~\cite{1986-Boesch} and its name is explained in the following. For each $G$ in $\mathcal{C}_{n,m}$ we define $\mu_i(G)=\binom{m}{i}-N_i^{(1)}(G)$ where $i\in \{0,1,\ldots,m\}$. Denote $\mu(G)=(\mu_0(G),\mu_1(G),\ldots,\mu_m(G))$. Two graphs $G$ and $H$ in $\mathcal{C}_{n,m}$ are \emph{equivalent} if $\mu(G)=\mu(H)$. The class of equivalent graphs in $\mathcal{C}_{n,m}$ equipped with the lexicographic order among tuples $\mu(G)$ is a partially ordered set. A minimum element, if any, is a $0$-element as defined by Boesch. 


If $G$ has at least $k$ vertices then its \emph{$k$-order edge connectivity of $G$}, denoted $\lambda^{(k)}(G)$, equals the minimum number of edges that must be removed to $G$ to obtain a spanning subgraph $H$ such that $\kappa(H)> k$. The maximization of the $1$-edge connectivity of a graph was studied by Harary~\cite{1962-Harary}. The second-order edge connectivity of $G$ was discussed by Goldsmith et al.~\cite{Goldsmith-1980}. 

For each $k$ in $\{1,2,\ldots,n\}$, $t_k(G)$ denotes the number of spanning forests in $G$ composed by $k$ trees. The number $t_1(G)$ is known as the \emph{tree-number} of $G$. A graph $G$ in $\mathcal{C}_{n,m}$ is called \emph{$t$-optimal} if $t_1(G)\geq t_1(H)$ for all $H$ in $\mathcal{C}_{n,m}$. The determination of $t$-optimal graphs in each nonempty class $\mathcal{C}_{n,m}$ is still unresolved; the reader can find a characterization of some $t$-optimal graphs in~\cite{2002-Petingi}.


Several graph invariants are evaluations of the Tutte or Whitney polynomial. As an example, let us derive the \emph{forest generating function} of a graph $G$ in 
$\mathcal{C}_{n,m}$ defined by $\sum_{i=0}^{n-1}t_{i+1}(G)x^i$. If $\mathcal{F}(G)$ denotes the set of all forests of $G$ then
\begin{equation*}
W_G(x,0)=\sum_{H \in \mathcal{F}(G)}x^{\kappa(H)-1}=\sum_{i=0}^{n-1}t_{i+1}(G)x^i, 
\end{equation*}
where we used that there are precisely $t_i(G)$ forests in $G$ composed by $i$ trees. Then, for each $i \in \{1,2,\ldots,n\}$, 
\begin{equation*}
t_i(G) = \dfrac{\partial^{i-1}W}{\partial x^{i-1}}(0,0).
\end{equation*}
In particular, the tree-number $t_1(G)$
of a graph $G$ equals $W_G(0,0)$. The reader can find a list of real-valued graph invariants that are evaluations of the Tutte-polynomial in the recent book~\cite{Book-Tutte-2022}. It is worth to remark that a great variety of graph invariants are either maximized or minimized among all graphs in $\mathcal{C}_{n,m}$ by Tutte-maximum graphs; see
~\cite{Kahl-2023} for further details. 





\section{Main result} \label{section:main}
Boesch~\cite{1986-Boesch} discussed the existence of a $0$-element in $\mathcal{C}_{n,m}$. 
Let us consider the following generalization. 
\begin{definition}
A graph $G$ in $\mathcal{C}_{n,m}$ is \emph{strong} if $N_i^{(k)}(G)\geq N_i^{(k)}(H)$ 
for each $H\in \mathcal{C}_{n,m}$ and each pair of integers $i \in \{0,1,\ldots,m\}$ and $k\in \{1,2,\ldots,n\}$. 
\end{definition}
By equation~\eqref{eq:rk}, each strong graph is $k$-UMRG for all $k\in \{1,2,\ldots,n\}$. As a consequence, the existence and construction of strong graphs is a relevant topic in network reliability in the wide sense as presented by Kelmans~\cite{1981-Kelmans}. In this section we will find an algebraic characterization of strong graphs. Consider the following definition inspired by the definition of Tutte-maximum graphs.

\begin{definition}
For each pair of graphs $G$ and $H$ in $\mathcal{C}_{n,m}$ we write $H\preceq_W G$ 
when $W_G(x,y)-W_{H}(x,y)=(1-xy)Q_H(x,y)$ for some nonnegative polynomial $Q_H(x,y)$. The graph $G$ is \emph{Whitney-maximum} if $H \preceq_W G$ for each $H$ in $\mathcal{C}_{n,m}$.    
\end{definition}

All the properties satisfied by Whitney-maximum graphs are also satisfied by Tutte-maximum graphs. In fact, the following lemma holds.
\begin{lemma}\label{lemma:TutteareWhitney}
Each Tutte-maximum graph is Whitney-maximum.    
\end{lemma}
\begin{proof}
Let $G$ be a Tutte-maximum graph in $\mathcal{C}_{n,m}$ and $H$ be any graph in $\mathcal{C}_{n,m}$. As $G$ is Tutte-maximum, there exists some nonnegative polynomial $P(x,y)$ such that $T_G(x,y)-T_{H}(x,y)=(x+y-xy)P(x,y)$. Consequently, 
\begin{align*}
W_G(x,y)-W_{H}(x,y)&= T_G(x+1,y+1)-T_{H}(x+1,y+1)\\ 
&= ((x+1)+(y+1)-(x+1)(y+1))P(x+1,y+1)\\
&= (1-xy)P(x+1,y+1).
\end{align*}
As $P(x,y)$ is nonnegative, $P(x+1,y+1)$ is also nonnegative. The lemma follows. 
\end{proof}

A natural question that arises is whether or not each Whitney-maximum graph is Tutte-maximum. Let $G$ and $H$ be the graphs depicted in  Figure~\ref{figure:GandH}. 

\begin{figure}[htb]
\begin{center}
\scalebox{1}{
\begin{tabular}{c}
\begin{tikzpicture}[scale=1]

\coordinate(a) at (0,0);
\coordinate(b) at (0,1);
\coordinate(c) at (0,2);
\coordinate(d) at (0,3);
\coordinate(e) at (2,0);
\coordinate(f) at (2,1);
\coordinate(g) at (2,2);
\coordinate(h) at (2,3);

\draw (a)--(e);
\draw (a)--(f);
\draw (a)--(g);
\draw (a)--(h);
\draw (b)--(e);
\draw (b)--(f);
\draw (b)--(g);
\draw (b)--(h);
\draw (c)--(e);
\draw (c)--(f);
\draw (c)--(g);
\draw (c)--(h);
\draw (d)--(e);
\draw (d)--(f);
\draw (d)--(g);
\draw (d)--(h);

\draw[bend right=-20, black] (a) to (b);
\draw[bend right=-20, black] (c) to (d);

\foreach \p in {a,b,c,d,e,f,g,h}
       \fill [black] (\p) circle (3pt);
\end{tikzpicture} 
\\
$G$
\end{tabular}
\begin{tabular}{c}
\begin{tikzpicture}[scale=1]
\coordinate(b) at (0,0);
\coordinate(c) at (1,0);
\draw[white] (b)--(c);
\end{tikzpicture}
\end{tabular}
\begin{tabular}{c}
\begin{tikzpicture}[scale=1]

\coordinate(a) at (0,0);
\coordinate(b) at (0,1);
\coordinate(c) at (0,2);
\coordinate(d) at (0,3);
\coordinate(e) at (2,0);
\coordinate(f) at (2,1);
\coordinate(g) at (2,2);
\coordinate(h) at (2,3);

\draw (a)--(e);
\draw (a)--(f);
\draw (a)--(g);
\draw (a)--(h);
\draw (b)--(e);
\draw (b)--(f);
\draw (b)--(g);
\draw (b)--(h);
\draw (c)--(e);
\draw (c)--(f);
\draw (c)--(g);
\draw (c)--(h);
\draw (d)--(e);
\draw (d)--(f);
\draw (d)--(g);
\draw (d)--(h);

\draw[bend right=-20, black] (c) to (d);
\draw[bend right=20, black] (g) to (h);

\foreach \p in {a,b,c,d,e,f,g,h}
       \fill [black] (\p) circle (3pt);
\end{tikzpicture} 
\\
$H$
\end{tabular}
}
\end{center}
\caption{Graphs $G$ and $H$. \label{figure:GandH}}
\end{figure}
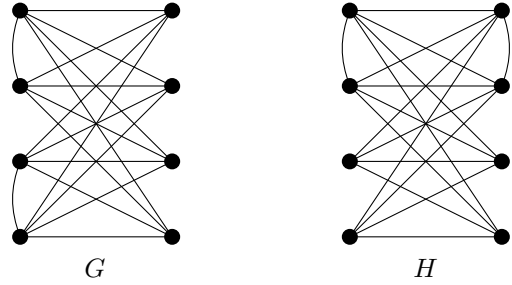

Computation shows that $G$ is the only Whitney-maximum graph in $\mathcal{C}_{8,18}$ up to isomorphism. However, 
\begin{equation*}
T_G(x,y)-T_H(x,y)=(x+y-xy)P(x,y),     
\end{equation*}
where $P(x,y)=4xy^5 + x^3y^2 + 4x^2y^3 + 12xy^4 + 2x^3y + 13x^2y^2 + 24xy^3 - x^3 + x^2y + 9xy^2 - 8y^3 - 4x^2 - 12xy - 19y^2 - 7x - 15y - 4$. As $P(x,y)$ has negative coefficients, $G$ is not Tutte-maximum disproving the converse of Lemma~\ref{lemma:TutteareWhitney}.

Now, let us study the relationship between strong graphs and Whitney-maximum graphs. The following technical lemma will be useful.
\begin{lemma}\label{lemma:NkG}
Let $G$ be any graph in $\mathcal{C}_{n,m}$. For each $k\in \{1,2,\ldots,n\}$ and 
each $i \in \{0,1,\ldots,m\}$,
\begin{equation*}
N_i^{(k)}(G) = \sum_{j=\max\{1,n-i\}}^{k}\frac{1}{(j-1)!(i-n+j)!}\frac{\partial^{2j+i-n-1} W_G}{\partial x^{j-1}y^{i-n+j}}(0,0).
\end{equation*}
\end{lemma}
\begin{proof}
Let $G$ be any graph in $\mathcal{C}_{n,m}$ and $i$ and $k$ as in the statement. Define $N_{i,j}(G)$ 
as the number of spanning graphs of $G$ composed precisely by $i$ edges and $j$ connected components. 
Observe that $N_i^{(k)}(G)=\sum_{j=1}^{k}N_{i,j}(G)$ and, further, if $j<n-i$ then $N_{i,j}(G)=0$. Therefore,
\begin{equation}\label{eq:reduced}
N_i^{(k)}(G) = \sum_{j=1}^{k}N_{i,j}(G) =  \sum_{j=\max\{1,n-i\}}^{k}N_{i,j}(G).   
\end{equation}
Let $j$ be any positive integer such that $j\geq n-i$. 
By definition, $N_{i,j}(G)$ equals the coefficient of $x^{j-1}y^{i-n+j}$ in $W_G(x,y)$. Consequently, 
\begin{equation}\label{eq:Taylor}
N_{i,j}(G) = \frac{1}{(j-1)!(i-n+j)!}\frac{\partial^{2j+i-n-1} W_G}{\partial x^{j-1}y^{i-n+j}}(0,0).
\end{equation}
The lemma follows replacing equation~\eqref{eq:Taylor} into \eqref{eq:reduced}. 
\end{proof}

\begin{lemma}\label{lemma:compareN}
If $G$ and $H$ are two graphs in $\mathcal{C}_{n,m}$ 
such that $H \preceq_W G$ then, for each $i \in \{0,1,\ldots,m\}$ and each $k\in \{1,2,\ldots,n\}$, 
$N_i^{(k)}(H) \leq N_i^{(k)}(G)$. 
\end{lemma}
\begin{proof}
Let $G$ and $H$ be two graphs in $\mathcal{C}_{n,m}$ and let $i$ and $k$ as in the statement. 
Define $Q(x,y)=1-xy$. 
As $H \preceq_W G$, there exists some nonnegative polynomial $P(x,y)$ such that $W_G(x,y)-W_{H}(x,y)=P(x,y)Q(x,y)$. Let $j$ be any integer in $\{1,2,\ldots,k\}$ such that $i-n+j\geq 0$. Observe that the second and higher order derivatives of $Q(x,y)$ with respect to $x$ are equal to $0$. Then, by Leibniz rule,  
\begin{align}
\frac{\partial^{j-1} (PQ)}{\partial x^{j-1}}(x,y) &= \sum_{\ell=0}^{j-1}\binom{j-1}{\ell}\frac{\partial^{\ell} P}{\partial x^{\ell}}(x,y)\frac{\partial^{j-1-\ell} Q}{\partial x^{j-1-\ell}}(x,y) \notag\\
&= \frac{\partial^{j-1} P}{\partial x^{j-1}}(x,y)Q(x,y)+(j-1)\frac{\partial^{j-2} P}{\partial x^{j-2}}(x,y)\frac{\partial Q}{\partial x}(x,y), \label{eq:exp}
\end{align}
where we use the convention that the second term on the right hand side of equation~\eqref{eq:exp} is $0$ when $j=1$. 

If we now derive $i-n+j$ times expression~\eqref{eq:exp} with respect to $y$ followed by an evaluation at $(x,y)=(0,0)$, 
\begin{align*}
 &\frac{\partial^{2j+i-n-1} (PQ)}{\partial x^{j-1}\partial y^{i-n+j}}(0,0)=\sum_{\ell=0}^{i-n+j}\binom{i-n+j}{\ell}\frac{\partial^{j-1+\ell} P}{\partial x^{j-1}\partial y^{\ell} }(0,0)\frac{\partial^{i-n+j-\ell} Q}{\partial y^{i-n+j-\ell}}(0,0)\\
 &+ (j-1)\sum_{\ell=0}^{i-n+j}\binom{i-n+j}{\ell}\frac{\partial^{j-2+\ell} P}{\partial x^{j-2}\partial y^{\ell}}(0,0)\frac{\partial^{i-n+j-\ell+1} Q}{\partial x \partial y^{i-n+j-\ell}}(0,0)\\
 &= \frac{\partial^{2j+i-n-1} P}{\partial x^{j-1}\partial y^{i-n+j}}(0,0)-(j-1)(i-n+j)\frac{\partial^{2j+i-n-3} P}{\partial x^{j-2}\partial y^{i-n+j-1}}(0,0),
\end{align*}
where we used Leibniz rule and the  facts that $\frac{\partial Q}{\partial y^s}(0,0)=0$ for each $s\geq 1$ and that $\frac{\partial Q}{\partial x \partial y}(0,0)=-1$. Now, by Lemma~\ref{lemma:NkG}, 
\begin{align*}
&N_i^{(k)}(G)-N_{i}^{(k)}(H)\\
&=\sum_{j=\max\{1,n-i\}}^{k}\frac{1}{(j-1)!(i-n+j)!}\frac{\partial^{2j+i-n-1} (W_G-W_H)}{\partial x^{j-1} \partial y^{i-n+j}}(0,0) \\
&= \sum_{j=\max\{1,n-i\}}^{k}\frac{1}{(j-1)!(i-n+j)!}\frac{\partial^{2j+i-n-1} (PQ)}{\partial x^{j-1} \partial y^{i-n+j}}(0,0)\\
&=\sum_{j=\max\{1,n-i\}}^{k}\frac{1}{(j-1)!(i-n+j)!}\cdot\\
&\left(\frac{\partial^{2j+i-n-1} P}{\partial x^{j-1}\partial y^{i-n+j}}(0,0)-(j-1)(i-n+j)\frac{\partial^{2j+i-n-3} P}{\partial x^{j-2}\partial y^{i-n+j-1}}(0,0)\right)\\
&= \frac{1}{(k-1)!(i-n+k)!}\frac{\partial^{2k+i-n-1} P}{\partial x^{k-1}\partial y^{i-n+k}}(0,0)\geq 0,
\end{align*}
where we identified a telescopic sum in the last equality and we used that $P(x,y)$ is nonnegative in the inequality. Consequently, $N_{i}^{(k)}(H) \leq N_i^{(k)}(G)$, as required. 
\end{proof}

\begin{lemma}\label{lemma:strongWhitney}
Each strong graph is Whitney-maximum.     
\end{lemma}
\begin{proof}
Let $G$ be a strong graph in $\mathcal{C}_{n,m}$ and let $G'$ be any graph in $\mathcal{C}_{n,m}$. The proof strategy is the following. First, we will construct a bijective mapping $\varphi$ from $\mathcal{S}(G)$ to $\mathcal{S}(G')$ such that 
$|E(\varphi(H))|=|E(H)|$ and $\kappa(\varphi(H)) \geq \kappa(H)$ 
for each $H$ in $\mathcal{S}(G)$. 
Then, we will prove that for each spanning subgraph $H$ in $\mathcal{S}(G)$ it holds that 
$x^{r(G)-r(H)}y^{c(H)}-x^{r(G')-r(\varphi(H))}y^{c(\varphi(H))}=(1-xy)P_H(x,y)$ for some nonnegative polynomial $P_H(x,y)$. The lemma will follow by summing the contributions of all spanning subgraphs $H$ in $\mathcal{S}(G)$ and the fact that the addition of nonnegative polynomials is also a nonnegative polynomial.

Now, we will construct the bijective mapping $\varphi$ satisfying the aforementioned conditions. Consider, for each $i\in \{0,1,\ldots,m\}$ and 
each $j,k\in \{1,2,\ldots,n\}$, the following sets,
\begin{align*}
\mathcal{S}_{i,j}(G)&=\{H:H\in \mathcal{S}(G), |E(H)|=i, \kappa(H)=j\},\\
\mathcal{S}_{i}^{(k)}(G)&=\{H:H\in \mathcal{S}(G), |E(H)|=i, \kappa(H)\leq k\}.
\end{align*}
The sets $\mathcal{S}_{i,j}(G')$ and $\mathcal{S}_{i}^{(k)}(G')$ are defined analogously (replacing $G$ by $G'$). Observe that $N_i^{(k)}(G)=|\mathcal{S}_{i}^{(k)}(G)|$  
and $\mathcal{S}_{i}^{(k)}(G) = \cup_{j=1}^{k}\mathcal{S}_{i,j}(G)$. 
As $G$ is strong, 
$N_i^{(1)}(G)\geq N_i^{(1)}(G')$ and the set $\mathcal{S}_{i}^{(1)}(G)$ has at least as many elements as $\mathcal{S}_{i}^{(1)}(G')$. Pick some set 
$\mathcal{G}_{i}^{(1)}$ in $\mathcal{S}_{i}^{(1)}(G)$ such that 
$|\mathcal{G}_{i}^{(1)}|=|\mathcal{S}_{i}^{(1)}(G')|$. Similarly, 
as $N_i^{(2)}(G)\geq N_i^{(2)}(G')$, the set $\mathcal{S}_{i}^{(2)}(G)$ has at least as many elements as $\mathcal{S}_{i}^{(2)}(G')$. 
Equivalently, $|\mathcal{S}_{i}^{(2)}(G)-\mathcal{G}_{i}^{(1)}| \geq |\mathcal{S}_{i}^{(2)}(G')-\mathcal{S}_{i}^{(1)}(G')|=|\mathcal{S}_{i,2}(G')|$. Consequently, there exists some set $\mathcal{G}_{i}^{(2)}$ in 
$\mathcal{S}_{i}^{(2)}(G)-\mathcal{G}_{i}^{(1)}$ such that $|\mathcal{G}_{i}^{(2)}|=|\mathcal{S}_{i,2}(G')|$. If we repeat this process we will construct, for each $i$ in $\{0,1,\ldots,m\}$, some disjoint sets 
$\mathcal{G}_{i}^{(1)},\ldots,\mathcal{G}_{i}^{(n)}$ such that 
$|\mathcal{G}_{i}^{(j)}|=|\mathcal{S}_{i,j}(G')|$ for each $j\in \{1,2,\ldots,n\}$, where $\mathcal{G}_{i}^{(j)} \subseteq \mathcal{S}_{i}^{(j)}(G)$. Further, observe that the sets $\mathcal{G}_{i}^{(1)},\ldots,\mathcal{G}_{i}^{(n)}$ are not only disjoint but also define a partition of 
$\mathcal{S}_i(G)$, since 
\begin{equation*}
|\mathcal{S}_i(G)| \geq \sum_{j=1}^{n}|\mathcal{G}_i^{(j)}|=\sum_{j=1}^{n} |\mathcal{S}_{i,j}(G')| = |\mathcal{S}_i(G')| = |\mathcal{S}_i(G)|     
\end{equation*}
thus the previous expression is a chain of equalities. Consequently, the set of spanning subgraphs in $G$ can be written as follows:   
$\mathcal{S}(G)=\cup_{i=0}^{m}\mathcal{S}_i(G)=\cup_{i=0}^{m}\cup_{j=1}^{n}\mathcal{G}_{i}^{(j)}$, and clearly, $\mathcal{S}(G')=\cup_{i=0}^{m}\cup_{j=1}^{n}\mathcal{S}_{i,j}(G')$. By construction, we know that $|\mathcal{G}_i^{(j)}|=|\mathcal{S}_{i,j}(G')|$ for each $i\in \{0,1,\ldots,m\}$ and $j\in \{1,2,\ldots,n\}$. Then, for 
each $i\in \{0,1,\ldots,m\}$ and $j\in \{1,2,\ldots,n\}$ 
there exists some bijective mapping $\varphi_{i,j}$ from $\mathcal{G}_i^{(j)}$ to $\mathcal{S}_{i,j}(G')$. Define the mapping $\varphi:\mathcal{S}(G)\to \mathcal{S}(G')$ such that, for each $H$ in $\mathcal{G}_{i}^{(j)}$,  
$\varphi(H)=\varphi_{i,j}(H)$. In the following we will denote $H'=\varphi(H)$.

Now, we will prove that for each $H$ in $\mathcal{S}(G)$, 
$\kappa(H')\geq \kappa(H)$. In fact, as $H$ belongs to $\mathcal{S}(G)$, there exists $i\in \{0,1,\ldots,m\}$ and $j\in \{1,2,\ldots,n\}$ such that $H \in \mathcal{G}_{i}^{(j)}$, and $H' \in \mathcal{S}_{i,j}(G')$. On the one hand,  $\mathcal{G}_{i}^{(j)}\subseteq \mathcal{S}_{i}^{(j)}$, and 
$\kappa(H)\leq j$. On the other hand, as $H' \in \mathcal{S}_{i,j}$, 
we have that $\kappa(H')=j$. Consequently, $\kappa(H')\geq \kappa(H)$.

Let us define, for each spanning graph $H$ in $\mathcal{G}_{i}^{(j)}$, 
the function $f_H(x,y)$ given by $x^{r(G)-r(H)}y^{c(H)}-x^{r(G')-r(H')}y^{c(H')}$.


By construction $|E(H)|=|E(H')|=i$ and $\kappa(H')\geq \kappa(H)$. 
Define $s$ as $\kappa(H')-\kappa(H)$. 
As both $G$ and $G'$ are connected on $n$ vertices, $r(G)=r(G')$ and  
$r(G')-r(H')=r(G)-r(H)+s$. Additionally, 
$c(H')=i-n+\kappa(H')=i-n+\kappa(H)+s=c(H)+s$. If $s=0$ 
then $f_H(x,y)=0$ and defining $P_H(x,y)=0$ we get that $f_H(x,y)=(1-xy)P(x,y)$ where $P(x,y)$ is nonnegative. Otherwise,
\begin{align*}
f_H(x,y)&=x^{r(G)-r(H)}y^{c(H)}-x^{r(G')-r(H')}y^{c(H')}\\
&= x^{r(G)-r(H)}y^{c(H)} (1-(xy)^s)=(x^{r(G)-r(H)}y^{c(H)}\sum_{\ell=0}^{s-1}(xy)^{\ell})(1-xy)\\
&=P_H(x,y)(1-xy), 
\end{align*}
where $P_H(x,y)=x^{r(G)-r(H)}y^{c(H)}\sum_{\ell=0}^{r-1}(xy)^{\ell}$ is a nonnegative polynomial. Then, 
\begin{equation*}
W_G(x,y)-W_{G'}(x,y)=\sum_{H\in \mathcal{S}(G)}f_H(x,y)=(\sum_{H\in \mathcal{S}(G)}P_H(x,y))(1-xy).    
\end{equation*}
As the addition of nonnegative polynomials is nonnegative, the lemma follows. 
\end{proof}

\begin{theorem}\label{theorem:equivalence}
A graph is Whitney-maximum if and only if it is strong. 
\end{theorem}
\begin{proof}
Let $G$ be a graph in $\mathcal{C}_{n,m}$. First, assume $G$ is Whitney-maximum. Let $H$ be any graph in $\mathcal{C}_{n,m}$. 
As $G$ is Whitney-maximum, $H \preceq_W G$. By Lemma~\ref{lemma:compareN} we know  that $N_i^{(k)}(G)\geq N_i^{(k)}(H)$ for each $i \in \{0,1,\ldots,m\}$ and each $k\in \{1,2,\ldots,n\}$. As $H$ is an arbitrary graph in $\mathcal{C}_{n,m}$ we conclude that $G$ is a strong graph. 
By Lemma~\ref{lemma:strongWhitney} we know that each strong graph is Whitney-maximum. The theorem follows. 
\end{proof}

\section{Applications} \label{section:consequences}
Remark~\ref{remark:invariants} gives closed forms for the $k$-order edge connectivity and the number of spanning forests composed by $k$ trees of $G$ as a function of its coefficients $N_i^{(k)}(G)$. 
\begin{remark}\label{remark:invariants}\,
If $k$ is a positive integer and $G$ is a graph in $\mathcal{C}_{n,m}$ such that $n\geq k$, then the following equalities hold
\begin{enumerate}[label=(\roman*)]
\item $\lambda^{(k)}(G)=\min\{x: x\geq 0, N_{m-x}^{(k)}(G)<\binom{m}{m-x}\}$,
\item $t_k(G)=N_{n-k}^{(k)}(G)$.
\end{enumerate}
\end{remark}

We are in position to prove the following result.
\begin{theorem}\label{theorem:main}
Each Whitney-maximum graph $G$ in $\mathcal{C}_{n,m}$ is 
$k$-UMRG for all $k\in \{1,2,\ldots,n\}$. For each $k\in \{1,2,\ldots,n\}$, the graph $G$ has the maximum $k$-edge order connectivity as well as the maximum number of spanning forests composed by $k$ trees. 
\end{theorem}
\begin{proof}
Let $G$ be any Whitney-maximum graph in $\mathcal{C}_{n,m}$. 
By Theorem~\ref{theorem:equivalence} we know that $G$ is strong. Therefore, by equation~\eqref{eq:rk} it follows that $G$ is $k$-UMRG for each 
$k\in \{1,2,\ldots,n\}$. 

Now, let $H$ be any graph in $\mathcal{C}_{n,m}$. As $G$ is Whitney-maximum, $H \preceq_W G$. Therefore,  Lemma~\ref{lemma:compareN} gives that $N_i^{(k)}(H)\leq N_i^{(k)}(G)$ for each $k\in \{1,2,\ldots,n\}$ and each $i \in \{0,1,\ldots,m\}$. The second sentence of the statement follows from Remark~\ref{remark:invariants}. The theorem follows. 
\end{proof}

Conjecture~\ref{conjecture1} is still unresolved. Nevertheless, such conjecture is true when restricted to graph classes $\mathcal{C}_{n,m}$ for which a Whitney-maximum graph exists. 
\begin{proposition}
If there exists some Whitney-maximum graph in $\mathcal{C}_{n,m}$ 
then each uniformly most reliable graph in $\mathcal{C}_{n,m}$ is a $0$-element in $\mathcal{C}_{n,m}$.    
\end{proposition}
\begin{proof}
Let $G$ be a Whitney-maximum graph in $\mathcal{C}_{n,m}$. By Theorem~\ref{theorem:main} we know that $G$ is strong. Therefore, $G$ is a $0$-element and uniformly most reliable as well. Let $H$ be any uniformly most reliable graph in $\mathcal{C}_{n,m}$. It is enough to prove that $H$ is a $0$-element. 
If $H$ is isomorphic to $G$ then we are done. Otherwise, as both graphs $G$ and $H$ are uniformly most reliable it follows that $R_{G}^{(1)}(p)=R_{H}^{(1)}(p)$ for all $p\in [0,1]$ and consequently, $N_i^{(1)}(G)=N_i^{(1)}(H)$ for each $i\in \{0,1,\ldots,m\}$. As $G$ is a $0$-element it follows that $H$ is also a $0$-element as required.      
\end{proof}

\section*{Acknowledgments}
This work is partially supported by City University of New York project entitled \emph{On the problem of characterizing graphs with maximum number of spanning trees}, grant number 66165-00. The author wants to thank Dr. Mart\'in Safe for his helpful  comments that improved the presentation of this manuscript.

\bibliographystyle{plain}

\end{document}